\documentclass[12pt]{article}
\usepackage{amssymb}
\usepackage{amsmath}
\usepackage{amsfonts}

\newtheorem{theorem}{Theorem}
\newtheorem{acknowledgement}[theorem]{Acknowledgement}

\newtheorem{corollary}[theorem]{Corollary}

\newtheorem{definition}[theorem]{Definition}

\newtheorem{lemma}[theorem]{Lemma}

\newtheorem{remark}[theorem]{Remark}

\newenvironment{proof}[1][Proof]{\noindent\textbf{#1.} }{\ \rule{0.5em}{0.5em}}
\input{tcilatex}

\begin{document}

\begin{center}
\textbf{{\Large {SPECIAL FUNCTIONS RELATED TO DEDEKIND TYPE DC-SUMS AND
THEIR APPLICATIONS}}}

\bigskip

{\Large 
}

Yilmaz SIMSEK

\bigskip

Akdeniz University, Faculty of Arts and Science, Department of Mathematics
07058 Antalya, Turkey

ysimsek@akdeniz.edu.tr

\bigskip

{\Large \ }
\end{center}

\textbf{{\Large {Abstract }}}In this paper we construct trigonometric
functions of the sum $T_{p}(h,k)$, which is called Dedekind type DC-(Dahee
and Changhee) sums. We establish analytic properties of this sum. We find
trigonometric representations of this sum. We prove reciprocity theorem of
this sums. Furthermore, we obtain relations between the Clausen functions,
Polylogarithm function, Hurwitz zeta function, generalized Lambert series ($G
$-series), Hardy-Berndt sums and the sum $T_{p}(h,k)$. We also give some
applications related to these sums and functions.

\bigskip



\section{Introduction, Definitions and Notations}

In this section, we give some definitions, notations and results related to
the Dedekind sums. Firstly we start with the definition of the classical
Dedekind sums.

Let%
\begin{equation*}
\left( \left( x\right) \right) =\left\{ 
\begin{array}{ll}
x-[x]_{G}-\frac{1}{2}\text{,} & \text{if }x\notin \mathbb{Z} \\ 
0\text{,} & \text{if }x\in \mathbb{Z}\text{,}%
\end{array}%
\right.
\end{equation*}%
$[x]_{G}$ being the largest integer $\leq x$. Let $h$ and $k$ be coprime
integers with $k>0$, the classical Dedekind sum $s(h,k)$ is defined as
follows%
\begin{equation*}
s\left( h,k\right) =\sum_{a=1}^{k-1}\left( \left( \frac{a}{k}\right) \right)
\left( \left( \frac{ha}{k}\right) \right) .
\end{equation*}

The reciprocity law of the classical Dedekind sums is given by%
\begin{equation*}
s(h,k)+s(k,h)=-\frac{1}{4}+\frac{1}{12}\left( \frac{h}{k}+\frac{k}{h}+\frac{1%
}{hk}\right) ,
\end{equation*}%
where $(h,k)=1$ and $h,k\in \mathbb{N}:=\left\{ 1,2,3,...\right\} $, and $%
\mathbb{N}_{0}:=\mathbb{N\cup }\left\{ 0\right\} $.

The classical Dedekind sums $s(h,k)$ arise in the transformation formula the
Dedekind-eta function. By using this transformation formula, Dedekind proved
reciprocity law of the classical Dedekind sums cf. \cite{Dedekind}. For
other proofs of reciprocity law of the classical Dedekind sums, see cf. (%
\cite{GroswaldRademacher}, \cite{HanceRademacher}, \cite{Apostol}, \cite%
{Apostol 1952}, \cite{Berndt-76}-\cite{Berndt1978}, \cite{Dieter}, \cite%
{Groswald}, \cite{S. Iseki}), see also the references cited in each of these
earlier works.

In the literature of the Dedekind sums, there are several generalizations of
the Dedekind sums that involve higher order Bernoulli functions and Euler
functions, the reader should look at \cite{GroswaldRademacher}, \cite%
{Apostol}, \cite{Berndt-76}-\cite{Berndt1978}, \cite{Dieter}, \cite{bayad2}, 
\cite{Beck} and \cite{TKim DC-sum} for references and, see also the
references cited in each of these earlier works.

In 1950, Apostol (\cite{Apostol}, \cite{Apostol 1952}) generalized Dedekind
sums as follows:%
\begin{equation}
S_{p}(h,k)=\sum_{a\func{mod}k}\frac{a}{k}\overline{B}_{p}(\frac{ah}{k}),
\label{b0-1}
\end{equation}%
where $(h,k)=1$ and $h,k\in \mathbb{N}$ and $\overline{B}_{p}(x)$ is the $p$%
-th Bernoulli function, which is defined as follows:%
\begin{align}
\overline{B}_{p}(x)& =B_{p}(x-[x]_{G})  \label{bo-2} \\
& =-p!\left( 2\pi i\right) ^{-p}\sum_{\QATOP{m=-\infty }{m\not=0}}^{\infty
}m^{-p}e^{2\pi imx},  \notag
\end{align}%
where $B_{p}(x)$ is the usual Bernoulli polynomials, which are defined by
means of the following generating function%
\begin{equation*}
\frac{te^{tz}}{e^{t}-1}=\sum_{n=0}^{\infty }B_{n}(z)\frac{t^{n}}{n!}\text{, }%
\left\vert t\right\vert <2\pi
\end{equation*}%
where $B_{n}(0)=B_{n}$ is denoted the Bernoulli number cf. \cite{Apostol}-%
\cite{S1}.

Observe that when $p=1,$ the sums $S_{1}(h,k)$ are known as the classical
Dedekind sums, $s(h,k)$.

The following theorem proved by Apostol \cite{Apostol}:

\begin{theorem}
\label{Thorem-1}Let $(h,k)=1$. For odd $p\geq 1$, we have 
\begin{equation}
S_{p}(h,k)=\frac{p!}{(2\pi i)^{p}}\sum_{%
\begin{array}{c}
m=1 \\ 
m\not\equiv 0(k)%
\end{array}%
}^{\infty }\frac{1}{m^{p}}(\frac{e^{2\pi imh/k}}{1-e^{2\pi imh/k}}-\frac{%
e^{-2\pi imh/k}}{1-e^{-2\pi imh/k}}).  \label{Eq-17}
\end{equation}
\end{theorem}

In \cite{Apostol 1952}, Apostol established a connection between the sums $%
S_{p}(h,k)$ and certain finite sums involving Hurwitz zeta functions. By
using this relation, he proved the reciprocity law of the sum $S_{p}(h,k)$.

By using same motivation of the Dedekind sums,in this paper, we study on
infinite series representation of the Dedekind type DC-sum, reciprocity law
of this sum and some special functions.

In \cite{TKimAAA2008} and \cite{TKimASCM2008DCeuler}, Kim defined the 
\textit{first kind} $n$-th Euler function $\overline{E}_{m}(x)$ as follows:%
\begin{equation}
\overline{E}_{m}(x)=\frac{2(m!)}{(\pi i)^{m+1}}\sum_{%
\begin{array}{c}
n=-\infty  \\ 
n\not=0%
\end{array}%
}^{\infty }\frac{e^{(2n+1)\pi ix}}{(2n+1)^{m+1}},  \label{1}
\end{equation}%
where $m\in \mathbb{N}$. Hoffman \cite{Hoffman} studied on Fourier series of
Euler polynomials. He also expressed the values of Euler polynomials at any
rational argument in terms of $\tan x$ and $\sec x$.

Observe that if $0\leq x<1$, then (\ref{1}) reduces to the \textit{first kind%
} $n$-th Euler polynomials $E_{n}(x)$ which are defined by means of the
following generating function%
\begin{equation}
\frac{2e^{tx}}{e^{t}+1}=\sum_{n=0}^{\infty }E_{n}(x)\frac{t^{n}}{n!},\text{ }%
\left\vert t\right\vert <\pi .  \label{1a}
\end{equation}

Observe that $E_{n}(0)=E_{n}$ denotes the first kind Euler number which is
given by the following recurrence formula%
\begin{equation}
E_{0}=1\text{ and }E_{n}=-\sum_{k=0}^{n}\binom{n}{k}E_{k},  \label{1aa1}
\end{equation}%
Some of them are given by $1$, $-\frac{1}{2}$, $0$, $\frac{1}{4}$, $\cdots $%
, $E_{n}=2^{n}E_{n}(\frac{1}{2})$\ and $E_{2n}=0$, ($n\in \mathbb{N}$)\ cf. (%
\cite{TKim2002RJMPVolkenborn}-\cite{TKim BernoulliDCeuler}, \cite{Hoffman}, 
\cite{OzdenSimsekAML}, \cite{simjnttwEul}, \cite{SimsekNA2009}, \cite%
{YsimsekASCM-2008}) and see also the references cited in each of these
earlier works.

In \cite{TKimASCM2008DCeuler} and \cite{TKimAAA2008}, by using Fourier
transform for the Euler function, Kim derived some formulae related to
infinite series and the first kind Euler numbers. For example, (\ref{1}), and%
\begin{equation}
\sum_{n=1}^{\infty }\frac{1}{(2n+1)^{2m+2}}=\frac{(-1)^{m+1}\pi
^{2m+2}E_{2m+1}}{4(2m+1)!}.  \label{1bb}
\end{equation}%
Kim \cite{TKimAAA2008} gave the following results:

\begin{subequations}
\begin{equation}
\sec hx=\frac{1}{\cosh x}=\frac{2e^{x}}{e^{2x}+1}=\sum_{n=0}^{\infty
}E_{n}^{\ast }\frac{t^{n}}{n!},\text{ }\left\vert t\right\vert <\frac{\pi }{2%
},  \label{1b}
\end{equation}%
where $E_{m}^{\ast }$\ is denoted the \textit{second kind} Euler numbers. By
(\ref{1a}) and (\ref{1b}), it is easy to see that 
\end{subequations}
\begin{equation*}
E_{m}^{\ast }=\sum_{n=0}^{m}\binom{m}{n}2^{n}E_{n},
\end{equation*}%
and%
\begin{equation*}
E_{2m}^{\ast }=-\sum_{n=0}^{m-1}\binom{2m}{2n}E_{2n}^{\ast }\text{ cf. \cite%
{TKimAAA2008}.}
\end{equation*}%
From the above $E_{0}^{\ast }=1$, $E_{1}^{\ast }=0$, $E_{2}^{\ast }=-1$, $%
E_{3}^{\ast }=0$, $E_{4}^{\ast }=5$,$\cdots $, and $E_{2m+1}^{\ast }=0$, ($%
m\in \mathbb{N}$).

The first and the second kind Euler numbers are also related to $\tan z$ and$%
\ \sec z$.

\begin{equation*}
\tan z=-i\frac{e^{iz}-e^{-iz}}{e^{iz}+e^{-iz}}=\frac{e^{2iz}}{2i}\left( 
\frac{2}{e^{2iz}+1}\right) -\frac{e^{-2iz}}{2i}\left( \frac{2}{e^{-2iz}+1}%
\right) .
\end{equation*}%
By using (\ref{1a}) and Cauchy product, we have%
\begin{eqnarray*}
\tan z &=&\frac{1}{2i}\sum_{n=0}^{\infty }E_{n}\frac{(2iz)^{n}}{n!}%
\sum_{n=0}^{\infty }\frac{(2iz)^{n}}{n!}-\frac{1}{2i}\sum_{n=0}^{\infty
}E_{n}\frac{(-2iz)^{n}}{n!}\sum_{n=0}^{\infty }\frac{(-2iz)^{n}}{n!} \\
&=&\frac{1}{2i}\sum_{n=0}^{\infty }\sum_{k=0}^{n}E_{k}\frac{(2iz)^{k}}{k!}%
\frac{(2iz)^{n-k}}{(n-k)!}-\frac{1}{2i}\sum_{n=0}^{\infty
}\sum_{k=0}^{n}E_{k}\frac{(-2iz)^{k}}{k!}\frac{(-2iz)^{n-k}}{(n-k)!} \\
&=&\frac{1}{2i}\sum_{n=0}^{\infty }\sum_{k=0}^{n}\frac{E_{k}}{k!(n-k)!}%
(2i)^{n}z^{n}-\frac{1}{2i}\sum_{n=0}^{\infty }\sum_{k=0}^{n}\frac{E_{k}}{%
k!(n-k)!}(-2i)^{n}z^{n} \\
&=&\sum_{j=0}^{\infty }(-1)^{n}2^{2j+1}\left( \sum_{k=0}^{2j+1}\binom{2j+1}{k%
}E_{k}\right) \frac{z^{2j+1}}{(2j+1)!}
\end{eqnarray*}

By using (\ref{1aa1}), we find that%
\begin{equation}
\tan z=\sum\limits_{n=0}^{\infty }(-1)^{n+1}\frac{2^{2n+1}E_{2n+1}}{(2n+1)!}%
z^{2n+1},\text{ }\left\vert z\right\vert <\pi .  \label{1c}
\end{equation}

\begin{remark}
The other proofs of (\ref{1c}) has also given the references cited in each
of these earlier work. In \cite{TKimAAA2008}, Kim gave another proof of (\ref%
{1c}). We shall give just a brief sketch as the details are similar to those
in \cite{TKimAAA2008}.%
\begin{eqnarray}
i\tan z &=&\frac{e^{iz}-e^{-iz}}{e^{iz}+e^{-iz}}  \label{1w} \\
&=&1-\frac{2}{e^{2iz}-1}+\frac{4}{e^{4iz}-1}.  \notag
\end{eqnarray}%
From the above%
\begin{equation*}
z\tan z=\sum\limits_{n=1}^{\infty }(-1)^{n}\frac{4^{n}(1-4^{n})B_{2n}}{(2n)!}%
z^{2n}.
\end{equation*}%
By using the above, we arrive at (\ref{1c}). Similarly Kim \cite{TKimAAA2008}
also gave the following relation:%
\begin{equation*}
\sec z=\sum\limits_{n=0}^{\infty }(-1)^{n}\frac{E_{2n}^{\ast }}{(2n)!}z^{2n}%
\text{, }\left\vert z\right\vert <\frac{\pi }{2}.
\end{equation*}
\end{remark}

Kim \cite{TKim DC-sum} defined \textit{the Dedekind type DC
(Daehee-Changhee) sums} as follows:

\begin{definition}
\label{Definition-1} Let $h$ and $k$ be coprime integers with $k>0$. Then%
\begin{equation}
T_{m}(h,k)=2\sum_{j=1}^{k-1}(-1)^{j-1}\frac{j}{k}\overline{E}_{m}(\frac{hj}{k%
}),  \label{2}
\end{equation}%
where $\overline{E}_{m}(x)$ denotes the $m$-th (first kind) Euler function.
\end{definition}

The sum $T_{m}(h,k)$ gives us same behavior of the Dedekind sums. Several
properties and identities of the sum $T_{m}(h,k)$ and Euler polynomials were
given by Kim \cite{TKim DC-sum}. By using these identities, Kim \cite{TKim
DC-sum} proved many theorems. The most fundamental result in the theory of
the Dedekind sums, Hardy-Berndt sums, Dedekind type DC and the other
arithmetical sums is the reciprocity law. The reciprocity law can be used as
an aid in calculating these sums.

The reciprocity law of the sum $T_{m}(h,k)$ is given as follows:

\begin{theorem}
(\cite{TKim DC-sum})\label{Theorem-2}Let $(h,k)=1$ and $h,k\in \mathbb{N}$
with $h\equiv 1\func{mod}2$ and $k\equiv 1\func{mod}2$. Then we have%
\begin{eqnarray*}
&&k^{p}T_{p}(h,k)+h^{p}T_{p}(k,h) \\
&=&2\sum\limits_{%
\begin{array}{c}
u=0 \\ 
u-[\frac{hu}{k}]\equiv 1\func{mod}2%
\end{array}%
}^{k-1}\left( kh(E+\frac{u}{k}+k(E+h-[\frac{hu}{k}])\right)
^{p}+(hE+kE)^{p}+(p+2)E_{p},
\end{eqnarray*}%
where%
\begin{equation*}
(hE+kE)^{n+1}=\sum_{j=1}^{n+1}\binom{n+1}{j}h^{j}E_{j}k^{n+1-j}E_{n+1-j}.
\end{equation*}
\end{theorem}

We summarize the result of this paper as follows:

In Section 2, we construct trigonometric representation of the sum $%
T_{p}(h,k)$. We give analytic properties of the sum $T_{p}(h,k)$.

In Section 3, we give some special functions and their relations. By using
these functions, we obtain relations between the sum $T_{p}(h,k)$, Hurwitz
zeta function, Lerch zeta function, Dirichlet series for the polylogarithm
function, Dirichlet's eta function and Clausen functions.

In Section 4, we prove reciprocity law of the sum $T_{p}(h,k)$.

In Section 5, we find relation between $G$-series (generalized Lambert
series) and the sums $T_{2y}(h,k)$.

In Section 6, we investigate relations between Hardy-Berndt sums, the sums $%
T_{2y}(h,k)$ and the other sums.

\section{Trigonometric Representation of the DC-sums}

In this section we can give relations between trigonometric functions and
the sum $T_{p}(h,k)$. We establish analytic properties of the sum $T_{p}(h,k)
$. We give trigonometric representation of the sum $T_{p}(h,k)$.

We now modify (\ref{1}) as follows:%
\begin{eqnarray}
\frac{(\pi i)^{m+1}}{2(m!)}\overline{E}_{m}(x) &=&\sum_{%
\begin{array}{c}
n=-\infty \\ 
n\neq 0%
\end{array}%
}^{\infty }\frac{e^{(2n+1)\pi ix}}{(2n+1)^{m+1}}  \label{1dyeni} \\
&=&\sum_{n=-\infty }^{-1}\frac{e^{(2n+1)\pi ix}}{(2n+1)^{m+1}}%
+\sum_{n=1}^{\infty }\frac{e^{(2n+1)\pi ix}}{(2n+1)^{m+1}}.  \notag
\end{eqnarray}%
From the above, we have%
\begin{equation}
\overline{E}_{m}(x)=\left\{ 
\begin{array}{c}
\frac{2(m!)}{(\pi i)^{m+1}}\sum_{n=1}^{\infty }\frac{\sin ((2n+1)\pi x)}{%
(2n+1)^{m+1}}\text{, if }m+1\text{ is odd} \\ 
\\ 
\frac{2(m!)}{(\pi i)^{m+1}}\sum_{n=1}^{\infty }\frac{\cos ((2n+1)\pi x)}{%
(2n+1)^{m+1}}\text{, if }m+1\text{ is even.}%
\end{array}%
\right.  \label{1d}
\end{equation}

If $m+1$ is even, then $m$ is odd, consequently, (\ref{1d}) reduces to the
following relation:

For $m=2y-1$, $y\in \mathbb{N}$,%
\begin{equation*}
\overline{E}_{2y-1}(x)=4(-1)^{y}\frac{(2y-1)!}{\pi ^{2y}}\sum_{n=1}^{\infty }%
\frac{\cos ((2n+1)\pi x)}{(2n+1)^{2y}}.
\end{equation*}%
If $m+1$ is odd, then $m$ is odd, hence (\ref{1d}) reduces to the following
relation:

For $m=2y$, $y\in \mathbb{N}$%
\begin{equation*}
\overline{E}_{2y}(x)=4(-1)^{y}\frac{(2y)!}{\pi ^{2y+1}}\sum_{n=1}^{\infty }%
\frac{\sin ((2n+1)\pi x)}{(2n+1)^{2y+1}}.
\end{equation*}%
Hence, from the above, we arrive at the following Lemma.

\begin{lemma}
\label{Lemma-1}Let $y\in \mathbb{N\diagdown }\left\{ 1\right\} $ and $0\leq
x\leq 1;$ $y=1$ and $0<x<1$. Then we have%
\begin{equation}
\overline{E}_{2y-1}(x)=\frac{(-1)^{y}4(2y-1)!}{\pi ^{2y}}\sum_{n=1}^{\infty }%
\frac{\cos ((2n+1)\pi x)}{(2n+1)^{2y}},  \label{2Y}
\end{equation}%
and%
\begin{equation}
\overline{E}_{2y}(x)=\frac{(-1)^{y}4(2y)!}{\pi ^{2y+1}}\sum_{n=1}^{\infty }%
\frac{\sin ((2n+1)\pi x)}{(2n+1)^{2y+1}}.  \label{3}
\end{equation}
\end{lemma}

In Lemma \ref{Lemma-1} substituting $0\leq x<1$, thus $\overline{E}_{2y-1}(x)
$ and $\overline{E}_{2y}(x)$ reduce to the Euler polynomials, which are
related to Clausen functions, given in Section 3, below.

We now modify the sum $T_{m}(h,k)$\ for odd and even integer $m$. Thus, by (%
\ref{2}), we define $T_{2y-1}(h,k)$ and $T_{2y}(h,k)$\ sums as follows:

\begin{definition}
\label{Definition2}Let $h$ and $k$ be coprime integers with $k>0$. Then%
\begin{equation}
T_{2y-1}(h,k)=2\sum_{j=0}^{k-1}(-1)^{j-1}\frac{j}{k}\overline{E}_{2y-1}(%
\frac{hj}{k}),  \label{4}
\end{equation}%
and%
\begin{equation}
T_{2y}(h,k)=2\sum_{j=0}^{k-1}(-1)^{j-1}\frac{j}{k}\overline{E}_{2y}(\frac{hj%
}{k}),  \label{5}
\end{equation}%
where $\overline{E}_{2y-1}(x)\ $and$\ \overline{E}_{2y}(x)$ denote the Euler
functions.
\end{definition}

By substituting equation (\ref{2Y}) into (\ref{4}), we have%
\begin{equation}
T_{2y-1}(h,k)=-\frac{8(-1)^{y}(2y-1)!}{k\pi ^{2y}}\sum_{j=1}^{k-1}(-1)^{j}j%
\sum_{n=1}^{\infty }\frac{\cos (\frac{\pi hj(2n+1)}{k})}{(2n+1)^{2y}}
\label{6}
\end{equation}%
From the above we have%
\begin{equation}
T_{2y-1}(h,k)=-\frac{8(-1)^{y}(2y-1)!}{k\pi ^{2y}}\sum_{n=1}^{\infty }\frac{1%
}{(2n+1)^{2y}}\sum_{j=1}^{k-1}(-1)^{j}j\cos \left( \frac{\pi hj(2n+1)}{k}%
\right) .  \label{1aa}
\end{equation}%
We next recall from \cite{Berndt-Goldberg} and \cite{Goldberg-PhD}\ that%
\begin{equation*}
\sum_{j=1}^{k-1}je^{\frac{(2n+1)\pi hij}{k}}=\left\{ 
\begin{array}{c}
\frac{k}{e^{\frac{(2n+1)\pi ih}{k}}-1}\text{ if }2n+1\not\equiv 0(k), \\ 
\\ 
\frac{k(k-1)}{2}\text{ if }2n+1\equiv 0(k).%
\end{array}%
\right. 
\end{equation*}%
From the above, it is easy to get%
\begin{equation*}
\sum_{j=1}^{k-1}(-1)^{j}je^{\frac{(2n+1)\pi hij}{k}}=\frac{k}{e^{\frac{%
(k+(2n+1)h)\pi i}{k}}-1}.
\end{equation*}%
By using an elementary calculations, we have%
\begin{equation}
\sum_{j=1}^{k-1}(-1)^{j}j\cos \left( \frac{(2n+1)\pi hj}{k}\right) =-\frac{k%
}{2}  \label{1e}
\end{equation}%
and%
\begin{equation}
\sum_{j=1}^{k-1}(-1)^{j}j\sin \left( \frac{(2n+1)\pi hj}{k}\right) =\frac{%
k\tan \left( \frac{\pi h(2n+1)}{2k}\right) }{2},  \label{1f}
\end{equation}%
where $2n+1\not\equiv 0(k)$. By substituting (\ref{1e}) into (\ref{1aa}) and
after some elementary calculations, we obtain%
\begin{equation*}
T_{2y-1}(h,k)=\frac{8(-1)^{y}(2y-1)!}{k\pi ^{2y}}\sum_{n=1}^{\infty }\frac{1%
}{(2n+1)^{2y}}.
\end{equation*}%
By substituting (\ref{1bb}) into the above, we easily arrive at the
following theorem.

\begin{theorem}
\label{Theorem-3}Let $y\in \mathbb{N}$, then we have%
\begin{equation*}
T_{2y-1}(h,k)=4E_{2y-1}.
\end{equation*}
\end{theorem}

By substituting equation (\ref{3}) into (\ref{5}), we have%
\begin{equation}
T_{2y}(h,k)=\frac{8(-1)^{y}(2y)!}{k\pi ^{2y+1}}\sum_{j=1}^{k-1}(-1)^{j}j%
\sum_{n=1}^{\infty }\frac{\sin (\frac{(2n+1)hj\pi }{k})}{(2n+1)^{2y+1}}
\label{1aa111}
\end{equation}%
By substituting (\ref{1f}) into the above, after some elementary
calculations, we arrive at the following theorem.

\begin{theorem}
\label{Theorem-4}Let $h$ and $k$ be coprime positive integers. Let $y\in 
\mathbb{N}$, then we have%
\begin{equation}
T_{2y}(h,k)=\frac{4(-1)^{y}(2y)!}{\pi ^{2y+1}}\sum_{\substack{ n=1 \\ %
2n+1\not\equiv 0(\func{mod}k)}}^{\infty }\frac{\tan (\frac{h\pi (2n+1)}{2k})%
}{(2n+1)^{2y+1}}.  \label{10}
\end{equation}
\end{theorem}

\section{DC-sums related to special functions}

In this section, we give relations between DC-sums and some special
functions.

In \cite{SrivastavaChoi}, Srivastava and Choi gave many applications of the
Riemann zeta function, Hurwitz zeta function, Lerch zeta function, Dirichlet
series for the polylogarithm function and Dirichlet's eta function. In \cite%
{Guillera and Sondow}, Guillera and Sandow obtained double integral and
infinite product representations of many classical constants, as well as a
generalization to Lerch's transcendent of Hadjicostas's double integral
formula for the Riemann zeta function, and logarithmic series for the
digamma and Euler beta functions. They also gave many applications. The 
\textit{Lerch trancendent} $\Phi (z,s,a)$\ (cf. e. g. \cite[p. 121 et seq.]%
{SrivastavaChoi}, \cite{Guillera and Sondow}) is the analytic continuation
of the series%
\begin{eqnarray*}
\Phi (z,s,a) &=&\frac{1}{a^{s}}+\frac{z}{(a+1)^{s}}+\frac{z}{(a+2)^{s}}%
+\cdots  \\
&=&\sum_{n=0}^{\infty }\frac{z^{n}}{(n+a)^{s}},
\end{eqnarray*}%
which converges for ($a\in \mathbb{C\diagdown Z}_{0}^{-}$, $s\in \mathbb{C}$
when $\left\vert z\right\vert <1$; $\Re (s)>1$ when $\left\vert z\right\vert
=1$) where as usual%
\begin{equation*}
\mathbb{Z}_{0}^{-}=\mathbb{Z}^{-}\cup \left\{ 0\right\} ,\ \mathbb{Z}%
^{-}=\left\{ -1,-2,-3,...\right\} .
\end{equation*}%
\ \ $\Phi $ denotes the familiar Hurwitz-Lerch Zeta function (cf. e. g. \cite%
[p. 121 et seq.]{SrivastavaChoi}). Relations between special function and
the function $\Phi $\ are given as follows \cite{Guillera and Sondow}:

Special cases include the analytic continuations of the Riemann zeta function%
\begin{equation*}
\Phi (1,s,1)=\zeta (s)=\sum_{n=1}^{\infty }\frac{1}{n^{s}}\text{, }\Re (s)>1,
\end{equation*}%
the Hurwitz zeta function%
\begin{equation*}
\Phi (1,s,a)=\zeta (s,a)=\sum_{n=0}^{\infty }\frac{1}{(n+a)^{s}}\text{, }\Re
(s)>1,
\end{equation*}%
the alternating zeta function (also called Dirichlet's eta function $\eta (s)
$)%
\begin{equation*}
\Phi (-1,s,1)=\zeta ^{\ast }(s)=\sum_{n=1}^{\infty }\frac{(-1)^{n-1}}{n^{s}},
\end{equation*}%
the Dirichlet beta function%
\begin{equation*}
\frac{\Phi (-1,s,\frac{1}{2})}{2^{s}}=\beta (s)=\sum_{n=0}^{\infty }\frac{%
(-1)^{n}}{(2n+1)^{s}},
\end{equation*}%
the Legendre chi function%
\begin{equation*}
\frac{z\Phi (z^{2},s,\frac{1}{2})}{2^{s}}=\chi _{s}(z)=\sum_{n=0}^{\infty }%
\frac{z^{2n+1}}{(2n+1)^{s}}\text{, (}\left\vert z\right\vert \leq 1;\Re (s)>1%
\text{),}
\end{equation*}%
the polylogarithm%
\begin{equation*}
z\Phi (z,n,1)=Li_{m}(z)=\sum_{n=0}^{\infty }\frac{z^{k}}{n^{m}}
\end{equation*}%
and the Lerch zeta function (sometimes called the Hurwitz-Lerch zeta
function)%
\begin{equation*}
L(\lambda ,\alpha ,s)=\Phi (e^{2\pi i\lambda },s,\alpha ),
\end{equation*}%
which is a special function and generalizes the Hurwitz zeta function and
polylogarithm cf. (\cite{ApostolLerch}, \cite{Berndt1972}, \cite{Choi}, \cite%
{Choi LCfunJKMS}, \cite{civicovic and klinokowski}, \cite{D.
CvijovicJMAA2007}, \cite{Guillera and Sondow}, \cite{Srivastava1999KobeJ}, 
\cite{SrivastavaCEMBRIDGE2000}, \cite{SrivastavaChoi}, \cite{TKimRJMP2007})
and see also the references cited in each of these earlier works.

By using (\ref{1dyeni}), we give relation between the Legendre chi function $%
\chi _{s}(z)$, and the function $\overline{E}_{m}(x)$ by the following
corollary:

\begin{corollary}
Let $m\in \mathbb{N}$. Then we have%
\begin{equation*}
\overline{E}_{m}(x)=\frac{2(m!)}{(\pi i)^{m+1}}\left( (-1)^{m+1}\chi
_{m+1}(e^{-\pi ix})+\chi _{m+1}(e^{\pi ix})\right) .
\end{equation*}
\end{corollary}

In \cite[p. 78, Theorem B]{SrivastavaCEMBRIDGE2000}, Srivastava proved the
following formulae which are related to Hurwitz zeta function, trigonometric
functions and Euler polynomials:%
\begin{equation*}
E_{2y-1}(\frac{p}{q})=(-1)^{y}\frac{4(2y-1)!}{(2q\pi )^{2y}}%
\sum_{j=1}^{q}\zeta (2y,\frac{2j-1}{q})\cos (\frac{\pi p(2j-1)}{q}),
\end{equation*}%
where $y,q\in \mathbb{N}$, $p\in \mathbb{N}_{0};$ $0\leq p\leq q$, and%
\begin{equation*}
E_{2y}(\frac{p}{q})=(-1)^{y}\frac{4(2y)!}{(2q\pi )^{2y+1}}%
\sum_{j=1}^{q}\zeta (2y+1,\frac{2j-1}{2q})\sin (\frac{\pi p(2j-1)}{q}),
\end{equation*}%
where $y,q\in \mathbb{N}$, $p\in \mathbb{N}_{0};$ $0\leq p\leq q$ and $\zeta
(s,x)$ denotes the Hurwitz zeta function. By substituting $p=0$ in the
above, then we have%
\begin{equation*}
E_{2y-1}(0)=(-1)^{y}\frac{4(2y-1)!}{(2q\pi )^{2y}}\sum_{j=1}^{q}\zeta (2y,%
\frac{2j-1}{q}).
\end{equation*}%
By using the above equation, we modify the sum $T_{2y-1}(h,k)$ as follows:

\begin{corollary}
Let $y,q\in \mathbb{N}$. Then we have%
\begin{equation*}
T_{2y-1}(h,k)=(-1)^{y}\frac{4(2y-1)!}{(2q\pi )^{2y}}\sum_{j=1}^{q}\zeta (2y,%
\frac{2j-1}{q}).
\end{equation*}
\end{corollary}

In \cite{ChoiSrivastavaademcik2003AMC}, Choi et al. gave relations between
the Clausen function, multiple gamma function and other functions. \textit{%
The higher-orde}r \textit{Clausen function }$Cl_{n}(t)$\textit{\ }(see \cite%
{SrivastavaChoi}, \cite[Eq-(4.15)]{ChoiSrivastavaademcik2003AMC}) be
defined, for all $n\in \mathbb{N\diagdown }\left\{ 1\right\} $, by%
\begin{equation*}
Cl_{n}(t)=\left\{ 
\begin{array}{c}
\sum_{k=1}^{\infty }\frac{\sin (kt)}{k^{n}}\text{ if }n\text{ is even,} \\ 
\\ 
\sum_{k=1}^{\infty }\frac{\cos (kt)}{k^{n}}\text{ if }n\text{ is odd.}%
\end{array}%
\right. 
\end{equation*}%
\ The following functions are related to the higher-order Clausen function
(cf. \cite{Srivastava1999KobeJ}, \cite[Eq-(5) and Eq-(6)]{D.
CvijovicJMAA2007})%
\begin{equation}
\mathsf{S}(s,x)=\sum_{n=1}^{\infty }\frac{\sin ((2n+1)x)}{(2n+1)^{s}}
\label{3-1a}
\end{equation}%
and%
\begin{equation}
\mathsf{C}(s,x)=\sum_{n=1}^{\infty }\frac{\cos ((2n+1)x)}{(2n+1)^{s}}.
\label{3-1b}
\end{equation}

In \cite{Srivastava1999KobeJ}, Srivastava studied on the functions $\mathsf{S%
}(s,x)$, $\mathsf{C}(s,x)$. When $x$ is a rational multiple of $2\pi $, he
gave the functions $\mathsf{S}(s,x)$, $\mathsf{C}(s,x)$\ in terms of Hurwitz
zeta functions. In \cite[Eq-(5) and Eq-(6)]{D. CvijovicJMAA2007}, Cvijovic
studied on the functions $\mathsf{S}(s,x)$, $\mathsf{C}(s,x)$. He gave many
applications of this function.

Espinosa and Moll \cite{Espinosa} gave relation between the functions $%
\mathsf{S}(s,x)$, $\mathsf{C}(s,x)$\ and $Cl_{n}(t)$ as follows:

For $0\leq q\leq 1$%
\begin{eqnarray*}
\mathsf{S}(2m+2,q) &=&Cl_{2m+2}(2\pi q), \\
\mathsf{S}(2m+1,q) &=&\frac{(-1)^{m+1}(2\pi )^{2m+1}}{2(2m+1)!}B_{2m+1}(q),
\end{eqnarray*}%
and%
\begin{eqnarray*}
\mathsf{C}(2m+1,q) &=&Cl_{2m+1}(2\pi q), \\
\mathsf{C}(2m+2,q) &=&\frac{(-1)^{m}(2\pi )^{2m+2}}{2(2m+2)!}B_{2m+2}(q).
\end{eqnarray*}

Setting $x=\frac{hj\pi }{k}$ and $s=2y$ and $s=2y+1$ in (\ref{3-1a}) and (%
\ref{3-1b}) respectively, than combine (\ref{6}) and (\ref{1aa111}), we
obtain the next corollary.

\begin{corollary}
Let $h$ and $k$ be coprime positive integers. Let $y\in \mathbb{N}$. Then we
have%
\begin{equation*}
T_{2y-1}(h,k)=-\frac{8(-1)^{y}(2y-1)!}{k\pi ^{2y}}\sum_{j=1}^{k-1}(-1)^{j}j%
\mathsf{C}(2y,\frac{hj\pi }{k}),
\end{equation*}%
and%
\begin{equation*}
T_{2y}(h,k)=\frac{8(-1)^{y}(2y)!}{k\pi ^{2y+1}}\sum_{j=1}^{k-1}(-1)^{j}j%
\mathsf{S}(2y+1,\frac{hj\pi }{k}).
\end{equation*}
\end{corollary}

Trickovic et al.\cite[p. 443, Eq-(3)]{Trickovic} gave relations between the
Clausen function and polylogarithm $Li_{\alpha }(z)$. They also gave the
fallowing relations: For $\alpha >0$%
\begin{equation}
\mathsf{S}(\alpha ,x)=\frac{i}{2}\left( \left( Li_{\alpha }(e^{-ix})-\frac{1%
}{2^{\alpha }}Li_{\alpha }(e^{-2ix})\right) -\left( Li_{\alpha }(e^{ix})-%
\frac{1}{2^{\alpha }}Li_{\alpha }(e^{2ix})\right) \right)  \label{3-1c}
\end{equation}%
and%
\begin{equation}
\mathsf{C}(\alpha ,x)=\frac{1}{2}\left( \left( Li_{\alpha }(e^{-ix})-\frac{1%
}{2^{\alpha }}Li_{\alpha }(e^{-2ix})\right) +\left( Li_{\alpha }(e^{ix})-%
\frac{1}{2^{\alpha }}Li_{\alpha }(e^{2ix})\right) \right) .  \label{3-1d}
\end{equation}

By substituting $x=\frac{hj\pi }{k}$ into (\ref{3-1d}) and (\ref{3-1c}); and
combine (\ref{6}) and (\ref{1aa111}), respectively; after some elementary
calculations, we easily find the next results.

\begin{corollary}
Let $h$ and $k$ be coprime positive integers. Let $y\in \mathbb{N}$. Then we
have%
\begin{eqnarray*}
T_{2y-1}(h,k) &=&-\frac{4(-1)^{y}(2y-1)!}{k\pi ^{2y}}%
\sum_{j=1}^{k-1}(-1)^{j}j\mathsf{\times } \\
&&\left( Li_{2y}(e^{-\frac{hij\pi }{k}})+Li_{2y}(e^{\frac{hij\pi }{k}})-%
\frac{Li_{2y}(e^{\frac{2hij\pi }{k}})+Li_{2y}(e^{-\frac{2hij\pi }{k}})}{%
2^{2y}}\right) ,
\end{eqnarray*}%
and%
\begin{eqnarray*}
T_{2y}(h,k) &=&\frac{4i(-1)^{y}(2y)!}{k\pi ^{2y+1}}\sum_{j=1}^{k-1}(-1)^{j}j%
\mathsf{\times } \\
&&\left( Li_{2y+1}(e^{-\frac{hij\pi }{k}})-Li_{2y+1}(e^{\frac{hij\pi }{k}})+%
\frac{Li_{2y+1}(e^{\frac{2hij\pi }{k}})-Li_{2y+1}(e^{-\frac{2hij\pi }{k}})}{%
2^{2y+1}}\right) .
\end{eqnarray*}
\end{corollary}

\section{Reciprocity Law}

The first proof of reciprocity law of the Dedekind sums does not contain the
theory of the Dedekind eta function related to Rademacher \cite{rademacher24}%
. The other proofs of the reciprocity law of the Dedekind sums were given by
Grosswald and Rademacher \cite{GroswaldRademacher}. Berndt \cite{Berndt-76}-%
\cite{Berndt-Goldberg} gave various types of Dedekind sums and their
reciprocity laws. Berndt's methods are of three types. The first method uses
contour integration which was first given by Rademacher \cite{rademacher24}.
This method has been used by many authors for example Isaki \cite{S. Iseki},
Grosswald \cite{Groswald}, Hardy \cite{Hardy}, his method is a different
technique in contour integration. The second method is the Riemann-Stieltjes
integral, which was invented by Rademacher \cite{HanceRademacher}. The third
method of Berndt is (periodic) Poisson summation formula. For the method and
technique see also the references cited in each of these earlier works.

The famous property of the all arithmetic sums is the reciprocity law. In
this section, by using contour integration, we prove reciprocity law of (\ref%
{10}). Our method is same as \cite{Berndt-76} and also for example cf. (\cite%
{rademacher24}, \cite{Berndt-76}-\cite{Berndt-Goldberg}, \cite{Groswald}, 
\cite{GroswaldRademacher}).

The initial different proof of the following reciprocity theorem is due to
Kim \cite{TKim DC-sum}, who first defined $T_{y}(h,k)$ sum.

\begin{theorem}
\label{Theorem-5}Let $h$, $k$, $y\in \mathbb{N}$ with $h\equiv 1\func{mod}2$
and $k\equiv 1\func{mod}2$ and $(h,k)=1$. Then we have%
\begin{eqnarray*}
&&kh^{2y+1}T_{2y}(h,k)+hk^{2y+1}T_{2y}(k,h) \\
&=&\frac{(-1)^{y}\pi ^{2y-1}\Gamma (2y+1)}{2\Gamma (4y+2)}E_{4y+1}+4\pi
^{2}(2y)!\sum_{a=0}^{y-1}\frac{E_{2a+1}E_{2y-2a-1}h^{2a+2}k^{2y-2a}}{%
(2a+1)!(2y-2a+1)!},
\end{eqnarray*}%
where $\Gamma (n+1)=n!$ and$\ E_{n}$ denote Euler gamma function and first
kind Euler numbers, respectively.
\end{theorem}

\begin{proof}
We shall give just a brief sketch as the details are similar to those in 
\cite[see Theorem 4.2]{Berndt-76}, \cite[ see Theorem 3]{Berndt-Goldberg}, 
\cite{Groswald} or \cite{GroswaldRademacher}. For the proof we use contour
integration method. So we define%
\begin{equation*}
F_{y}(z)=\frac{\tan \pi hz\tan \pi kz}{z^{2y+1}}.
\end{equation*}%
Let $C_{N}$ be a positive oriented circle of radius $R_{N}$, with $\ 1\leq
N<\infty $, centred at the origin. Assume that the sequence of radii $R_{N}$
is increasing to $\infty $. $R_{N}$ is chosen so that the circles always at
a distance greater than some fixed positive integer number from the points $%
\frac{m}{2h}$ and $\frac{n}{2k}$, where $m$ and $n$ are integers.

Let%
\begin{equation*}
I_{N}=\frac{1}{2\pi i}\dint\limits_{C_{N}}\frac{\tan \pi hz\tan \pi kz}{%
z^{2y+1}}dz.
\end{equation*}%
From the above, we get%
\begin{equation*}
I_{N}=\frac{1}{2\pi }\dint\limits_{0}^{2\pi }\frac{\tan \left( \pi
hR_{N}e^{i\theta }\right) \tan \left( \pi kR_{N}e^{i\theta }\right) }{\left(
R_{N}e^{i\theta }\right) ^{2y}}d\theta .
\end{equation*}%
By $C_{N}$, if $R_{N}\rightarrow \infty $, then $\tan \left( R_{N}e^{i\theta
}\right) $ is bounded. Consequently, we easily see that%
\begin{equation*}
\lim_{N\rightarrow \infty }I_{N}=0\text{ as }R_{N}\rightarrow \infty .
\end{equation*}%
Thus, on the interior $C_{N}$, the integrand of $I_{N}$ that is$\ F_{y}(z)$\
has simple poles at $z_{1}=\frac{2m+1}{2h}$, $-\infty <m<\infty $, and $%
z_{2}=\frac{2n+1}{2k}$, $-\infty <n<\infty $. If we calculate the residues
at the $z_{1}$ and $z_{2}$, we easily obtain respectively as follows%
\begin{equation*}
-\frac{2^{2y+1}k^{2y}}{\pi (2m+1)^{2y+1}}\tan (\frac{(2m+1)\pi h}{2k})\text{%
, }-\infty <m<\infty 
\end{equation*}%
and%
\begin{equation*}
-\frac{2^{2y+1}h^{2y}}{\pi (2n+1)^{2y+1}}\tan (\frac{(2n+1)\pi k}{2h})\text{%
, }-\infty <n<\infty .
\end{equation*}%
If $h$ and $k$ is odd integers, then $F_{y}(z)$\ has double poles at $z_{3}=%
\frac{2j+1}{2}$,\ $-\infty <j<\infty $. Thus the residue is easily found to
be%
\begin{equation*}
-\frac{(2y+1)2^{2y+1}}{2(2j+1)^{4y+2}\pi ^{2}hk}\text{,}\ -\infty <j<\infty .
\end{equation*}%
The integrand of $I_{N}$ has pole of order $2y+1$ at $z_{4}=0$, $y\in 
\mathbb{N}$. Recall the familiar Taylor expansion of $\tan z$\ in (\ref{1c}%
). By straight-forward calculation, we find the residues at the $z_{4}$ as
follows%
\begin{equation*}
(-1)^{y}\left( 2\pi \right) ^{2y+2}\sum_{a=1}^{y-1}\frac{%
E_{2a+1}E_{2y-2a-1}h^{2a+1}k^{2y-2a-1}}{(2a+1)!(2j-2a-1)!}.
\end{equation*}%
Now we are ready to use residue theorem, hence we find that%
\begin{eqnarray*}
I_{N} &=&-\frac{2^{2y+1}h^{2y}}{\pi }\sum_{\mid \frac{2m+1}{2h}\mid <R_{N}}%
\frac{\tan (\frac{(2m+1)\pi k}{2h})}{(2m+1)^{2y+1}}-\frac{2^{2y+1}k^{2y}}{%
\pi }\sum_{\mid \frac{2n+1}{2k}\mid <R_{N}}\frac{\tan (\frac{(2n+1)\pi k}{2h}%
)}{(2n+1)^{2y+1}} \\
&&-\frac{(2y+1)2^{2y}}{\pi ^{2}hk}\sum\limits_{j=-\infty }^{\infty }\frac{1}{%
(2j+1)^{4y+2}}+(-1)^{y}\left( 2\pi \right) ^{2y+2}\sum_{a=0}^{y-1}\frac{%
E_{2a+1}E_{2y-2a-1}h^{2a+1}k^{2y-2a-1}}{(2a+1)!(2y-2a-1)!}.
\end{eqnarray*}%
By using (\ref{1bb}) and letting $N\rightarrow \infty $\ into the above,
after straight-forward calculations, we arrive at the desired result.
\end{proof}

\begin{remark}
We also recall from \cite[pp. 20, Eq-(11.2)-(11-3)]{HansRademacherKITAP}\
that%
\begin{equation}
\tan z=\sum_{k=1}^{\infty }\mathcal{T}_{k}\frac{z^{2k-1}}{(2k-1)!},
\label{1ccc}
\end{equation}%
where%
\begin{equation*}
\mathcal{T}_{k}=(-1)^{k-1}\frac{B_{2k}}{(2k)}(2^{2k}-1)2^{2k}.
\end{equation*}%
The integrand of $I_{N}$ has pole of order $2y+1$ at $z_{4}=0$, $y\in 
\mathbb{N}$. Recall the familiar Taylor expansion of $\tan z$\ in (\ref{1ccc}%
). By straight-forward calculation, we find the residues at the $z_{4}$ as
follows%
\begin{equation*}
\pi ^{2y}\sum_{a=0}^{y+1}\frac{\mathcal{T}_{a}\mathcal{T}_{y-a+1}}{%
(2a-1)!(2y-2a-1)!}h^{2a-1}k^{2y-2a+1}.
\end{equation*}%
Thus we modify Theorem \ref{Theorem-5} as follows:%
\begin{eqnarray*}
&&kh^{2y+1}T_{2y}(h,k)+hk^{2y+1}T_{2y}(k,h) \\
&=&\frac{(-1)^{y}\pi ^{2y-1}\Gamma (2y+1)}{2\Gamma (4y+2)}E_{4y+1}+\frac{%
\left( -1\right) ^{y}(2y)!}{4^{y}}\sum_{a=0}^{y+1}\frac{\mathcal{T}_{a}%
\mathcal{T}_{y-a+1}}{(2a-1)!(2y-2a-1)!}h^{2a-1}k^{2y-2a+1}.
\end{eqnarray*}
\end{remark}

We now give relation between Hurwitz zeta function, $\tan z$\ and the sum $%
T_{2y}(h,k)$.

Hence, substituting $n=rk+j$, $0\leq r\leq \infty $, $1\leq j\leq k$ into (%
\ref{10}), and recalling that $\tan (\pi +\alpha )=\tan \alpha $, then we
have 
\begin{eqnarray*}
T_{2y}(h,k) &=&\frac{4(-1)^{y}(2y)!}{\pi ^{2y+1}}\sum_{j=1}^{k}\sum_{r=0}^{%
\infty }\frac{\tan (\frac{\pi h2(rk+j)+1}{2k})}{(2(rk+j)+1)^{2y+1}} \\
&=&\frac{4(-1)^{y}(2y)!}{\pi ^{2y+1}(2k)^{2y+1}}\sum_{j=1}^{k}\tan (\frac{%
\pi h(2j+1)}{2k})\sum_{r=0}^{\infty }\frac{1}{(r+\frac{2j+1}{2k})^{2y+1}} \\
&=&\frac{4(-1)^{y}(2y)!}{\pi ^{2y+1}(2k)^{2y+1}}\sum_{j=1}^{k}\tan (\frac{%
\pi h(2j+1)}{2k})\zeta (2y+1,\frac{2j+1}{2k})
\end{eqnarray*}%
where $\zeta (s,x)$ denotes the Hurwitz zeta function. Thus we arrive at the
following theorem:

\begin{theorem}
Let $h$ and $k$ be coprime positive integers. Let $y\in \mathbb{N}$. Then we
have%
\begin{equation}
T_{2y}(h,k)=\frac{4(-1)^{y}(2y)!}{(2k\pi )^{2y+1}}\sum_{j=1}^{k}\tan (\frac{%
\pi h(2j+1)}{2k})\zeta (2y+1,\frac{2j+1}{2k}).  \label{10A}
\end{equation}
\end{theorem}

\section{$G$-series(Generalized Lambert series) related to DC-sums}

The main purpose of this section is to give relation between $G$-series and
the sums $T_{2y}(h,k)$.

By using (\ref{1w}), we have%
\begin{equation}
i\tan z=\frac{e^{2iz}}{1+e^{2iz}}-\frac{e^{-2iz}}{1+e^{-2iz}}.  \label{1w1}
\end{equation}%
We recall in \cite{simjnt2003dede} that%
\begin{equation}
\frac{e^{2iz}}{1+e^{2iz}}=i\tan z+\frac{e^{-2iz}}{1+e^{-2iz}}.  \label{1w2}
\end{equation}

Hence setting $2iz=\frac{h\pi i(2n+1)}{k}$, with $(h,k)=1$, $n\in \mathbb{N}$
in (\ref{1w1}) with (\ref{10}), we obtain the following corollary:

\begin{corollary}
Let $h$ and $k$ be coprime positive integers. Let $y\in \mathbb{N}$, then we
have%
\begin{equation}
T_{2y}(h,k)=\frac{4i(-1)^{y+1}(2y)!}{(2k\pi )^{2y+1}}\sum_{\substack{ n=1 \\ %
2n+1\not\equiv 0(\func{mod}k)}}^{\infty }\frac{1}{(2n+1)^{2y+1}}\left( \frac{%
e^{\frac{h\pi i(2n+1)}{k}}}{1+e^{\frac{h\pi i(2n+1)}{k}}}-\frac{e^{-\frac{%
h\pi i(2n+1)}{k}}}{1+e^{-\frac{h\pi i(2n+1)}{2k}}}\right) ,  \label{1w3}
\end{equation}%
where $i=\sqrt{-1}$.
\end{corollary}

The above corollary give us the sums $T_{2y}(h,k)$ are related to $G$-series.

In \cite{Trahan}, Trahan defined the $G$\textit{-series (or Generalized
Lambert series)} as follows:%
\begin{equation}
G(z)=\sum \frac{a_{n}z^{n}}{1+c_{n}z^{n}},  \label{1w4}
\end{equation}%
where the coefficients $a_{n}$ and $c_{n}$ are complex numbers and $%
a_{n}c_{n}\neq -1$. A $G$-series is a \textit{power series} if, for all $n$, 
$c_{n}=0$ and a \textit{Lambert series} if, if, for all $n$, $c_{n}=-1$. In
the literature a $G$-series is usually considered as a generalized Lambert
series. The Lambert series, first studied by J. H. Lambert, is analytic at
the origin and has a power series expansion at\ the origin. For $\mid z\mid
<1$ J. H. Lambert found that%
\begin{equation*}
\sum_{n=1}^{\infty }\frac{z^{n}}{1-z^{n}}=\sum_{n=1}^{\infty }\tau
_{n}z^{n}=z+2z^{2}+2z^{3}+3z^{4}+2z^{5}+4z^{6}+...,
\end{equation*}%
where $\tau _{n}$\ is the number of divisors of $n$ (cf., e.g., \cite{Trahan}%
) and see also the references cited in each of these earlier works.

\begin{theorem}
\label{Theorem 7}(\cite[p. 29, Theorem A and Theorem B]{Trahan}) a) If $\mid
z\mid <\frac{1}{\overline{\lim }\sqrt[n]{\mid c_{n}\mid }}$, then the $G$%
-series $\sum \frac{a_{n}z^{n}}{1+c_{n}z^{n}}$ converges if and only if the
power series $\sum a_{n}z^{n}$ confects.

b) If $\mid z\mid >\frac{1}{\underline{\lim }\sqrt[n]{\mid c_{n}\mid }}$ and 
$c_{n}\neq 0$, then the $G$-series $\sum \frac{a_{n}z^{n}}{1+c_{n}z^{n}}$
converges if and only if the power series $\sum \frac{a_{n}}{c_{n}}$
converges.
\end{theorem}

\begin{theorem}
\label{Theorem 8}(\cite[p. 30, Theorem 2 and Theorem 3]{Trahan}) a) Assume $%
\mid z\mid <\frac{1}{\overline{\lim }\sqrt[n]{\mid a_{n}\mid }}$. If there
is no subsequence of $\left\{ c_{n}z^{n}\right\} $\ which has limit $-1$,
then the $G$-series $\sum \frac{a_{n}z^{n}}{1+c_{n}z^{n}}$ converges
absolutely.

b) If $\mid z\mid >\frac{1}{\underline{\lim }\sqrt[n]{\mid c_{n}\mid }}$ and 
$c_{n}\neq 0$\ for all $n$, then the $G$-series $\sum \frac{a_{n}z^{n}}{%
1+c_{n}z^{n}}$ converges (absolutely) if and only if the series $\sum \frac{%
a_{n}}{c_{n}}$ converges (absolutely).
\end{theorem}

By setting $a_{n}=\frac{1}{n^{2y+1}}$, $y\in \mathbb{N}$, $c_{n}=1$ in (\ref%
{1w4}) and using Theorem \ref{Theorem 7} and Theorem \ref{Theorem 8}, we
obtain the fallowing relation:%
\begin{equation}
G(e^{\pi iz})-2^{-2y-1}G(e^{2\pi iz})=\sum_{n=1}^{\infty }\frac{1}{%
(2n+1)^{2y+1}}\left( \frac{e^{(2n+1)\pi iz}}{1+e^{(2n+1)\pi iz}}\right) ,
\label{1w5}
\end{equation}%
and%
\begin{equation}
G(e^{-\pi iz})-2^{-2y-1}G(e^{-2\pi iz})=\sum_{n=1}^{\infty }\frac{1}{%
(2n+1)^{2y+1}}\left( \frac{e^{-(2n+1)\pi iz}}{1+e^{-(2n+1)\pi iz}}\right) .
\label{1w6}
\end{equation}%
By substituting (\ref{1w5}) and (\ref{1w6}) into (\ref{1w3}), we arrive at
the following theorem. The next theorem give us relation between $T_{2y}(h,k)
$ sum and $G$-series.

\begin{theorem}
Let $h$ and $k$ be coprime positive integers. If $y\in \mathbb{N}$, then we
have%
\begin{equation*}
T_{2y}(h,k)=\frac{4i(-1)^{y+1}(2y)!}{(2k\pi )^{2y+1}}\left( G(e^{\frac{h\pi
i)}{k}})-G(e^{-\frac{h\pi i)}{k}})+\frac{G(e^{-\frac{2h\pi i}{k}})-G(e^{%
\frac{2h\pi i}{k}})}{2^{2y+1}}\right) .
\end{equation*}
\end{theorem}

\begin{remark}
We recall from \cite{simjnt2003dede}, \cite{simADCMgenDedeEisn}\ and \cite%
{SimUMJ2004} that relations between Dedekind sums, Hardy-Berndt sums and
Lambert series were given by the author and see also the references cited in
each of these earlier works.
\end{remark}

\section{Some Applications}

In (\ref{10}) if $h$ and $k$ are odd and $y=0$, then $T_{2y}(h,k)$ reduces
to the Hardy-Berndt sum $S_{5}(h,k)$.

\QTP{ite}
Recently Hardy sums (Hardy-Berndt sums) have been studied by many
Matematicians (\cite{Hardy}, \cite{Berndt-Goldberg}, \cite{Goldberg-PhD}, 
\cite{S1}, \cite{simjkms2006}, \cite{simjnt2003dede}) and see also the
references cited in each of these earlier works. Hardy-Berndt sum $%
s_{5}(h,k) $ is defined as follows:

\QTP{ite}
Let $h$ and $k$ be integers with $(h,k)=1$. Then%
\begin{equation}
s_{5}(h,k)=\sum\limits_{j=1}^{k}(-1)^{j+[\frac{hj}{k}]_{G}}((\frac{j}{k})).
\label{Eq-2}
\end{equation}%
From the above, recall from \cite{Berndt-Goldberg} that, we have%
\begin{equation}
s_{5}(h,k)=\sum\limits_{j=1}^{k}(-1)^{j}\frac{j}{k}(-1)^{[\frac{hj}{k}]_{G}}
\label{1aa11}
\end{equation}%
By using the well-known Fourier expansion%
\begin{equation*}
(-1)^{[x]_{G}}=\frac{4}{\pi }\sum_{n=0}^{\infty }\frac{\sin ((2n+1)\pi x)}{%
2n+1}\text{ cf. (\cite{Berndt-Goldberg}, \cite{Goldberg-PhD})}
\end{equation*}%
into (\ref{1aa11}), we get%
\begin{equation*}
s_{5}(h,k)=\frac{4}{k\pi }\sum_{n=0}^{\infty }\frac{1}{2n+1}%
\sum\limits_{j=1}^{k}(-1)^{j}j\sin (\frac{(2n+1)\pi hj}{k}).
\end{equation*}%
By substituting (\ref{1aa111}) into the above, we immediately find the
following result.

\begin{lemma}
\label{Lemma.2}Let $h$ and $k$ are odd with $(h,k)=1$. Then we have%
\begin{equation*}
T_{0}(h,k)=2s_{5}(h,k).
\end{equation*}
\end{lemma}

By using Lemma \ref{Lemma.2} and Theorem \ref{Theorem-4}, we arrive at the
following theorem.

\begin{theorem}
\label{Theorem6}Let $h$ and $k$ are odd with $(h,k)=1$. Then we have%
\begin{equation*}
\mathcal{S}_{5}(h,k;y)=\frac{T_{2y}(h,k)}{2}.
\end{equation*}
\end{theorem}

\begin{remark}
Substituting $y=0$ into Theorem \ref{Theorem6}, we get $s_{5}(h,k)=2\mathcal{%
S}_{5}(h,k;0)$. Consequently, the sum $T_{2y}(h,k)$ give us generalized
Hardy-Berndt sum $s_{5}(h,k)$.
\end{remark}

In \cite{YsimsekTJM2009}, the author defined that%
\begin{equation*}
Y(h,k)=4ks_{5}(h,k),
\end{equation*}%
where $h$ and $k$ are odd with $(h,k)=1$. Thus from Lemma \ref{Lemma.2}, we
have the following corollary.

\begin{corollary}
Let $h$ and $k$ are odd with $(h,k)=1$. Then we have%
\begin{equation*}
T_{0}(h,k)=\frac{Y(h,k)}{2k}.
\end{equation*}%
Observe that the sum $T_{2y}(h,k)$ also give us generalization of the sum $%
Y(h,k)$.
\end{corollary}

\begin{remark}
Elliptic Apostol-Dedekind sums have been studied by many authors. Bayad \cite%
{bayad2}, constructed multiple elliptic Dedekind sums as an elliptic
analogue of Zagier's sums multiple Dedekind sums. In \cite{SimsekKooKim},
Simsek et al. defined elliptic analogue of the Hardy sums. By using same
method in \cite{bayad2}, elliptic analogue of the sum $T_{m}(h,k)$ may be
defined. In this paper, we do not study on elliptic analogue of the sum $%
T_{m}(h,k)$. By using $p$-adic $q$-Volkenborn integral, in \cite%
{TKim2001Dedekind} and \cite{TKim2002RJMPVolkenborn}, Kim defined $p$-adic $%
q $-Dedekind sums. In \cite{simjkms2006}, \cite{simjmaaqDed}, \cite%
{YsimsekTJM2009}, \cite{SimsekNA2009}, we defined $q$-Dedekind type sums, $q$%
-Hardy-Berndt type sums and $p$-adic $q$-Dedekind sums. By using same
method, $p$-adic $q$-analogue of the sum $T_{m}(h,k)$ may be defined.
\end{remark}

\begin{acknowledgement}
This paper was supported by the Scientific Research Project Administration
of Akdeniz University.
\end{acknowledgement}

\end{document}